\documentclass[11pt,reqno]{amsart}
\usepackage{amsmath,amsthm,amssymb,bm,amscd}
\usepackage[hypertex]{hyperref}

\makeatletter
    
    \@addtoreset{equation}{section}
\makeatother

\newtheorem{Thm}{Theorem}[section]
\newtheorem{Lem}[Thm]{Lemma}
\newtheorem{Cor}[Thm]{Corollary}
\newtheorem{Prop}[Thm]{Proposition}
\newtheorem{Rem}[Thm]{Remark}

\newcommand{\C}{\mathbb{C}}           
\newcommand{\Z}{\mathbb{Z}}
\newcommand{\Q}{\mathbb{Q}}

\newcommand{\ad}{\text{ad}}

\newcommand{\ch}{\mathrm{ch}\,}

\newcommand{\fa}{{\mathfrak a}}             
\newcommand{\fb}{{\mathfrak b}}
\newcommand{\fg}{{\mathfrak g}}
\newcommand{\fh}{{\mathfrak h}}

\newcommand{\fn}{{\mathfrak n}}
\newcommand{\fp}{{\mathfrak p}}

\newcommand{\hfg}{\hat{\fg}}
\newcommand{\hfh}{\hat{\fh}}
\newcommand{\hfb}{\hat{\fb}}

\newcommand{\hfn}{\hat{\fn}}
\newcommand{\hfp}{\hat{\fp}}
\newcommand{\hfm}{\hat{\fm}}
\newcommand{\hP}{\hat{P}}

\newcommand{\hI}{\hat{I}}
\newcommand{\hW}{\hat{W}}

\newcommand{\ga}{\alpha}

\newcommand{\gl}{\lambda}
\newcommand{\gL}{\Lambda}
\newcommand{\gd}{\delta}
\newcommand{\gD}{\Delta}

\newcommand{\gt}{\theta}

\newcommand{\gs}{\sigma}
\newcommand{\gS}{\Sigma}

\renewcommand{\hat}{\widehat}
\newcommand{\ol}{\overline}

\newcommand{\hV}{\hat{V}}

\newcommand{\msl}{\mathfrak{sl}}

\newcommand{\bfa}{f_{\ga}}

\renewcommand{\hfm}{\hat{\mathfrak{m}}}
\newcommand{\ff}{\mathfrak{f}}

\begin{document}

\title{Defining relations of fusion products and Schur positivity}
\author{Katsuyuki Naoi}
%
\address{%
Katsuyuki \textsc{Naoi} \\
Institute of Engineering \\
Tokyo University of Agriculture and Technology\\
2-24-16 Naka-cho, Koganei-shi, Tokyo 184-8588, JAPAN\\
e-mail: naoik@cc.tuat.ac.jp}

%
%

\maketitle

\begin{abstract}
 In this note we give defining relations of an $\mathfrak{sl}_{n+1}[t]$-module defined by the fusion product of simple 
 $\mathfrak{sl}_{n+1}$-modules whose highest weights are multiples of a given fundamental weight.
 From this result we obtain a surjective homomorphism between two fusion products,
 which can be considered as a current algebra analog of Schur positivity.
\end{abstract}

\section{Introduction}

Let $\fg = \mathfrak{sl}_{n+1}(\C)$ with index set $I=\{1,\ldots,n\}$, and fix a triangular decomposition
$\fg = \fn_+ \oplus \fh \oplus \fn_-$.
Denote by $\varpi_i$ ($i \in I$) the fundamental weights.
Let $\fg[t]=\fg\otimes \C[t]$ be the associated current algebra.
For $m \in I$ and a sequence $\bm{\ell}=(\ell_1, \ell_2,\ldots,\ell_p)$ of nonnegative integers,
we define a $\fg[t]$-module $V_m(\bm{\ell})$ by
\[ V_m(\bm{\ell}) = V(\ell_1 \varpi_m) *V(\ell_2\varpi_m) * \cdots * V(\ell_p\varpi_m).
\]
Here $V(\gl)$ is the simple $\fg$-module with highest weight $\gl$,
and $*$ denotes the fusion product defined by Feigin and Loktev in \cite{MR1729359}.
We may assume without loss of generality that $\ell_1 \ge \ell_2 \ge \cdots \ge \ell_p$, that is, $\bm{\ell}$ is a partition.

In \cite{MR3296163}, Chari and Venkatesh have introduced a large family of indecomposable $\fg[t]$-modules 
(with $\fg$ a general simple Lie algebra) indexed by a sequence of partitions, in terms of generators and relations.
In this note, we will show that the fusion product $V_m(\bm{\ell})$ is isomorphic to a module belonging to their family.
More explicitly, we show the following defining relations of $V_m(\bm{\ell})$.\\[-5pt]

\par
\noindent
{\bfseries Theorem.}\ \ 
{\itshape 
 Let $m \in I$ and $\bm{\ell} = (\ell_1\ge\cdots\ge\ell_p)$ be a partition. 
 Set $L_i = \ell_i +\cdots +\ell_{p-1}+\ell_p$ for $1 \le i \le p$ and $L_i = 0$ for $i>p$.
 Then $V_m(\bm{\ell})$ is isomorphic to the 
 $\fg[t]$-module generated by a vector $v$ with relations
 \begin{align*}
  \fn_+[t]v=0, \ \ \ (h\otimes t^s) v&= \gd_{s0}L_1\langle h,\varpi_m\rangle v \text{ for } 
  h \in \fh, \ s \in \Z_{\ge 0},\nonumber\\
  (f_\ga\otimes \C[t])v&= 0 \text{ for } \ga \in \gD_+ \text{ with } \langle h_\ga,\varpi_m\rangle = 0,\nonumber\\
  f_\ga^{L_1+1}v&=0 \text{ for } \ga \in \gD_+ \text{ with } \langle h_\ga, \varpi_m\rangle =1, \nonumber\\
  (e_\ga\otimes t)^{s}f_\ga^{r+s}v&=0 \text{ for } \ga \in \gD_+, r,s \in \Z_{>0} \text{ with } \langle h_\ga, \varpi_m 
   \rangle = 1,\nonumber\\
  &\hspace{40pt} r+s \ge 1+kr + L_{k+1} \text{ for some } k\in\Z_{>0}.
 \end{align*}
 Here $\gD_+$ is the set of positive roots, $ h_\ga$ is the coroot corresponding to $\ga$, and $e_\ga$ and $f_\ga$ are 
 root vectors corresponding to $\ga$ and $-\ga$ respectively.}
\par

This theorem for $\fg = \mathfrak{sl}_2$ has been proved in \cite{MR1988973} and \cite{MR3296163}.
In the case $p=2$, this can be found in \cite{MR3336341} and \cite{Fourier} (see also \cite{CSVW}).

Let us introduce a motivation of the theorem.
For that we consider the case $\fg = \mathfrak{sl}_2(\C)$ for a moment.
Let $(\ell_1\ge \ell_2), (r_1\ge r_2)$ be partitions of a positive integer $\ell$.
By the well-known Clebsch-Gordan formula 
\[ V(\ell\varpi_1)\otimes V(r\varpi_1) = V\big(|\ell-r|\varpi_1\big) \oplus \cdots \oplus 
   V\big((\ell+r-2)\varpi_1\big) \oplus V\big((\ell+r)\varpi_1\big),
\]
we see that there exists a surjective $\fg$-module homomorphism 
\[ V(\ell_1\varpi_1)\otimes V(\ell_2\varpi_1) \twoheadrightarrow V(r_1\varpi_1) \otimes V(r_2\varpi_1)
\]
if and only if $\ell_2 \ge r_2$.
This surjection implies that the difference of their characters can be written as a sum of characters of simple $\fg$-modules.
Since the characters of simple $\fg$-modules are known as Schur functions, 
this property is called \textit{Schur positivity}.
Generalization of the surjection to a more general $\fg$ and more general $\fg$-modules has been studied in 
\cite{MR2358610, MR2369890, MR3248990,MR3226992}.
In particular when $\fg = \mathfrak{sl}_{n+1}$, it follows from \cite{MR3248990} (see also \cite{MR2369890})
that for $m \in I$ and two partitions $(\ell_1\ge\cdots\ge \ell_p)$, $(r_1\ge\cdots \ge r_p)$ of a positive integer $\ell$,
there exists a surjective $\fg$-module homomorphism 
\begin{equation}\label{eq:surjection}
 V(\ell_1\varpi_m) \otimes \cdots \otimes V(\ell_p\varpi_m) \twoheadrightarrow V(r_1\varpi_m) \otimes \cdots \otimes 
 V(r_p\varpi_m)
\end{equation}
if $\ell_i + \cdots +\ell_p \geq r_i + \cdots +r_p$ holds for each $1\le i \le p$.

Fourier and Hernandez have raised the following question in the introduction of \cite{MR3226992}: Can surjections
such as (\ref{eq:surjection}) be obtained from surjective $\fg[t]$-module homomorphisms between the 
corresponding fusion products?
(Recall that the $\fg$-module structures of a tensor product and a fusion product are the same.)
By inspecting the defining relations of the theorem we obtain the following corollary,
which gives a positive answer to their question in our setting.

\begin{Cor}
 Let $m \in I$, and $\bm{\ell}=(\ell_1 \geq \cdots \geq \ell_p)$, $\bm{r}=(r_1 \geq \cdots \geq r_p)$ be two partitions of a 
 positive integer $\ell$.
 We assume that $\ell_i + \cdots +\ell_p \geq r_i + \cdots +r_p$ holds for each $1 \le i \le p$.
 Then there exists a surjective $\fg[t]$-module homomorphism
 from $V_m(\bm{\ell})$ onto $V_m(\bm{r})$.
\end{Cor}

It would be an interesting problem to generalize the theorem to a more general $\fg$ or more general modules.
These will be studied elsewhere.

The organization of this paper is as follows.
We fix basic notations in Subsection \ref{Subsection:preliminary},
and recall the definition of fusion products in Subsection \ref{Subsection:Fusion_product}.
By \cite{MR2964614} $V_m(\bm{\ell})$ can be realized as a $\fg[t]$-submodule of a module over the affine Lie algebra $\hfg$,
which is recalled in Subsection \ref{Subsection:Realization}.
In Subsection \ref{Subsection:Chari-Venkatesh}, we recall some results in \cite{MR3296163} needed for the proof of 
the main theorem, and show one technical lemma.
Then we prove the theorem in Section \ref{Section:proof} by determining the defining relations recursively
using the realization, in which we apply the method used in \cite{MR3120578}.

\section{Preliminaries}

\subsection{Simple Lie algebra, current algebra, and affine Kac-Moody Lie algebra of type 
{\boldmath$A$}}\label{Subsection:preliminary}

Let $\fg = \mathfrak{sl}_{n+1}(\C)$ with index set $I=\{1,\ldots,n\}$.
We fix a triangular decomposition $\fg = \fn_+ \oplus \fh \oplus \fn_-$.
Let $\ga_i$ and $\varpi_i$ ($i\in I$) be simple roots and fundamental weights respectively.
We use the labeling in \cite[Section 4.8]{MR1104219}.
For convenience we set $\varpi_0=0$.
Let $\gD$ be the root system, $\gD_+$ the set of positive roots, $W$ the Weyl group with simple reflections 
$\{s_i\mid i\in I\}$ and longest element $w_0$,
and $(\ , \ )$ the unique non-degenerate
$W$-invariant symmetric bilinear form on $\fh^*$ satisfying $(\ga,\ga) = 2$ for all $\ga \in \gD$.
Let 
\[ \gt = \ga_1 + \cdots + \ga_{n-1}+\ga_n
\]
be the highest root.
For each $\ga \in \gD$, let $ h_\ga$ be its coroot, $\fg_\ga$ the corresponding root space,
and $e_{\ga} \in \fg_\ga$ a root vector satisfying $[e_\ga,e_{-\ga}] = h_\ga$.
We also use the notations $f_\ga = e_{-\ga}$ for $\ga \in \gD_+$, $h_i = h_{\ga_i}$, $e_i = e_{\ga_i}$ and $f_i = f_{\ga_i}$.
Denote by $P$ the weight lattice, by $P_+$ the set of dominant integral weights, 
and by $V(\gl)$ ($\gl \in P_+$) the simple $\fg$-module with highest weight $\gl$.
For $i \in I$, set 
\[ i^* = n+1-i \in I.
\]
Note that $w_0(\varpi_i) = -\varpi_{i^*}$ holds.

Given a Lie algebra $\fa$, its \textit{current algebra} $\fa[t]$ is defined by the tensor product $\fa \otimes \C[t]$
equipped with the Lie algebra structure given by 
\[ [x\otimes f(t), y \otimes g(t)] = [x,y] \otimes f(t)g(t).
\]
For $k\in \Z_{> 0}$, let $t^k\fa[t]$ denote the ideal $\fa \otimes t^k\C[t] \subseteq \fa \otimes \C[t]$.

Let $\hfg = \fg\otimes \C[t,t^{-1}] \oplus \C c \oplus \C d$ be the nontwisted affine Lie algebra associated with $\fg$.
Here $c$ is the canonical central element and $d$ is the degree operator.
Note that $\fg$ and $\fg[t]$ are naturally considered as Lie subalgebras of $\hfg$.
Let $\hI = I \sqcup \{0\}$, and define Lie subalgebras $\hfh$, $\hfn_+$, and $\hfb$ as follows:
\[ \hfh = \fh \oplus \C c \oplus \C d, \ \ \hfn_+ = \fn_+ \oplus t\fg[t], \ \ 
   \hfb = \hfh \oplus \hfn_+.
\]
We also define $\hfh_{\mathrm{cl}} = \fh \oplus \C c$.
We often consider $\fh^*$ (resp.\ $\hfh_{\mathrm{cl}}^*$) as a subspace of $\hfh^*$ by setting 
\[ \langle c,\gl\rangle = \langle d,\gl\rangle = 0 \text{ for }  \gl \in \fh^* \ \ \ \big(\text{resp.}\ 
   \langle d,\gl\rangle = 0 \text{ for }  \gl \in \hfh_{\mathrm{cl}}^*\big).
\]
Let $\hat{\gD}$ be the root system of $\hfg$, $\hP$ the weight lattice, $\hP_+$ the set of dominant integral weights, 
and $\hat{W}$ the Weyl group with simple reflections $\{s_i \mid i \in \hI\}$.
Denote by $\gd\in \hP$ the null root, and by $\gL_0 \in \hP$ the unique element satisfying
\[ \langle \fh,\gL_0\rangle = \langle d, \gL_0 \rangle = 0, \ \ \ \langle c, \gL_0 \rangle = 1.
\]
Let $\ga_0 = \gd -\gt$, $e_0 = f_{\gt} \otimes t$ and $f_0 = e_{\gt} \otimes t^{-1}$.
Given an integrable $\hfg$-module $M$ and $i \in \hI$, define a linear automorphism $\Phi_i^M$ on $M$ by
\[ \Phi^M_i = \mathrm{exp}(f_i)\mathrm{exp}(-e_i)\mathrm{exp}(f_i),
\]
see \cite[Lemma 3.8]{MR1104219}.
For each $w\in \hW$ fix a reduced expression $w=s_{i_k}\cdots s_{i_1}$, and set
$\Phi^M_w=\Phi^M_{i_k}\cdots \Phi^M_{i_1}$. 
Then $\Phi^M_w$ satisfies
\begin{equation*}
 \Phi^M_w(M_\mu) = M_{w(\mu)} \ \text{ for } \mu \in \hP, \ \ \ \Phi_w^M \hfg_\ga (\Phi_w^M)^{-1}
  = \hfg_{w(\ga)} \text{ for } \ga \in \hat{\gD}.
\end{equation*}
In particular by considering the adjoint representation, an algebra automorphism on $U(\hfg)$ 
is defined for each $w \in \hW$, which is denoted by $\Phi_w$.
Note that $\Phi^M_w$ for $w \in W$ is also defined on a finite-dimensional $\fg$-module $M$.

Define $t_\gl \in \mathrm{GL}(\hfh^*)$ for $\gl \in P$ by
\begin{equation*}
  t_{\gl}(\mu) = \mu + \langle c,\mu \rangle \gl - \big((\mu, \gl) + 
  \frac{1}{2}(\gl, \gl)\langle  c, \mu \rangle\big)\gd,
\end{equation*}
see \cite[Chapter 6]{MR1104219}.
Let $T(P) = \{t_\gl \mid \gl \in P\}$ and $\widetilde{W} = W \ltimes T(P)$, 
which is called the \textit{extended affine Weyl group}.
Here $w \in W$ and $t_\gl \in T(P)$ satisfy $wt_\gl w^{-1}=t_{w(\gl)}$.
For $i \in \hI$, let 
\begin{equation*}
 \tau_i = t_{\varpi_i}w_{i,0}w_0 \in \widetilde{W}
\end{equation*}
where $w_{i,0}$ is the longest element of $W_{\varpi_i}$,
the stabilizer of $\varpi_i$ in $W$.
We have 
\begin{equation}\label{eq:rule_of_tau}
 \tau_i(\gd) = \gd, \ \ \ \tau_i(\ga_j) = \ga_{\ol{i+j}}, \ \text{ and } \ \tau_i(\varpi_j+\gL_0) \equiv \varpi_{\ol{i+j}}+
  \gL_0 \ \text{ mod } \Q\gd \ \text{ for } j \in \hI
\end{equation}
where $\ol{i+j} \equiv i+j$ mod $n+1$.
Set $\gS = \{\tau_i \mid i\in \hI\}$.
It is known that $\widetilde{W} = \hW\rtimes \gS $.
We define an action of $\gS$ on $\hfg$ by letting $\tau_i$ act as a Lie algebra automorphism determined by 
\begin{align*}
 \tau_i(e_j) = e_{\ol{i+j}}, \ \ \tau_i(f_j) = f_{\ol{i+j}} \text{ for }  j \in \hI, \ \
 \langle \tau_i(h), \tau_i(\gl)\rangle = \langle h,\gl\rangle \text{ for } h \in \hfh, \gl \in \hfh^*.
\end{align*}

\subsection{Fusion product}\label{Subsection:Fusion_product}

Let us recall the definition of the fusion product introduced in \cite{MR1729359}.
Note that $U(\fg[t])$ has a natural $\Z_{\ge 0}$-grading defined by
\[ U(\fg[t])^k = \{X \in U(\fg[t]) \mid [d,X] = kX\}.
\]
Let $\gl_1,\ldots,\gl_p$ be a sequence of elements of $P_+$, and $c_1,\ldots,c_p$ pairwise distinct complex numbers.
We define a $\fg[t]$-module structure on $V(\gl_i)$ as follows:
\[ \big(x\otimes f(t)\big) v = f(c_i)xv \ \text{ for } x \in \fg, f(t) \in \C[t], v \in V(\gl_i).
\]
Denote this $\fg[t]$-module by $V(\gl_i)_{c_i}$.
Let $v_i$ be a highest weight vector of $V(\gl_i)$. 
Then the $\fg[t]$-module $V(\gl_1)_{c_1} \otimes \cdots \otimes V(\gl_p)_{c_p}$ is 
generated by $v_1 \otimes \cdots \otimes v_p$ (see \cite{MR1729359}), 
and the grading on $U(\fg[t])$ induces a filtration on $V(\gl_1)_{c_1}\otimes \cdots \otimes V(\gl_p)_{c_p}$ by
\[ \Big(V(\gl_1)_{c_1}\otimes \cdots \otimes V(\gl_p)_{c_p}\Big)^{\leq k} = \sum_{r \leq k} 
   U(\fg[t])^r(v_1\otimes \cdots \otimes v_p).
\]
Now the $\fg[t]$-module obtained by taking the associated graded is denoted by 
\[ V(\gl_1) * \cdots * V(\gl_p),
\] 
and called the \textit{fusion product} of $V(\gl_1),\ldots,V(\gl_p)$.
Though the definition depends on the parameters $c_i$, we omit them for the notational convenience.
All fusion products appearing in this paper do not depend on the parameters up to isomorphism.
Note that, by definition, we have
\[ V(\gl_1) * \cdots * V(\gl_p) \cong V(\gl_1) \otimes \cdots \otimes V(\gl_p)
\]
as a $\fg$-module.

\subsection{Another realization of fusion products}\label{Subsection:Realization}

Kirillov-Reshetikhin modules for $\fg[t]$ are $\fg[t]$-modules defined in terms of generators and relations,
which have been introduced in \cite{MR2238884}.
In \cite{MR2964614} the fusion products of Kirillov-Reshetikhin modules for $\fg[t]$
were studied when $\fg$ is of type $ADE$, and a new realization of these modules using Joseph functors was given.
When $\fg$ is of type $A$, a Kirillov-Reshetikhin module is just the evaluation module at $t=0$ of $V(k\varpi_i)$ 
with $k\in \Z_{> 0}$ and $i \in I$,
and hence their fusion products are what we are studying in this note. 
In this subsection we will reformulate the result of \cite{MR2964614} in type $A$ in a different way
(see Remark \ref{Rem}).
This formulation has previously been used in \cite{MR3120578}, and is more suitable for later use since we can apply Lemma 
\ref{Lem:proceed} stated below.

First we introduce several notations.
Assume that $V$ is a $\hfg$-module and $D$ is a $\hfb$-submodule of $V$.
For $\tau \in \gS$, denote by $F_\tau V$ the pull-back $(\tau^{-1})^*V$ with respect to the Lie algebra 
automorphism $\tau^{-1}$ on $\hfg$, and define a $\hfb$-submodule $F_\tau D \subseteq F_\tau V$ in the obvious way.
For $i \in \hI$ let $\hfp_i$ denote the parabolic subalgebra $\hfb \oplus \C f_i \subseteq \hfg$,
and set $F_iD= U(\hfp_i)D \subseteq V$ to be the $\hfp_i$-submodule generated by $D$. 
Finally we define $F_w D$ for $w \in \widetilde{W}$ as follows:
let $\tau \in \gS$ and $w' \in \hat{W}$ be the elements such that $w=w'\tau$, and choose a reduced 
expression $w'=s_{i_k} \cdots s_{i_1}$. 
Then we set
\[ F_{w}D = F_{i_k}\cdots F_{i_1} F_\tau D\subseteq F_\tau V.
\]
Though the definition depends on the choice of the expression of $w'$, we use $F_w$ by an abuse of notation.

For $\gL \in \hP_+$ let $\hV(\gL)$ be the simple highest weight $\hfg$-module with highest weight $\gL$.
Denote by $\C_{\gL}$ the $1$-dimensional $\hfb$-submodule of $\hV(\gL)$ spanned 
by a highest weight vector.
Note that $F_\tau\hV(\gL) \cong \hV(\tau\gL)$ and $F_\tau \C_{\gL} \cong \C_{\tau\gL}$ for $\tau \in \gS$.
Let 
\begin{equation*}
 \hfb' = \hfb\cap \fg[t]= \fh \oplus \hfn_+.
\end{equation*}
Now \cite[Theorem 6.1]{MR2964614} is reformulated as follows.

\begin{Thm}\label{Thm:realization}
 Let $\bm{\ell} = (\ell_1 \ge \cdots\ge\ell_p)$ be a partition, and $m_1,\ldots,m_p$ a sequence of elements 
 of $I$.
 As a $\hfb'$-module, we have
 \begin{align*}
  V(\ell_1\varpi_{m_1}) * \cdots& * V(\ell_p\varpi_{m_p})\\
  &\cong F_{t_{-\varpi_{m_1^*}}}\Big(\C_{(\ell_1-\ell_2)\gL_0} \otimes \cdots
  \otimes F_{t_{-\varpi_{m_{p-1}^*}}}\big(\C_{(\ell_{p-1}-\ell_p)\gL_0} \otimes F_{t_{-\varpi_{m_p^*}}}\C_{\ell_p\gL_0}\big)\!
  \cdots \!\Big).
 \end{align*}
\end{Thm}

\begin{Rem}\label{Rem}\normalfont
 In \cite[Theorem 6.1]{MR2964614} the right-hand side is defined in terms of Joseph functors, but it can easily be proved to 
 be isomorphic to the right-hand side of Theorem \ref{Thm:realization} as follows.
 By the universality of Joseph functors, there exists a surjection between two modules.
 Moreover their characters coincide by \cite[Corollary 6.2]{MR2964614} and \cite[Theorem 5]{MR1887117},
 and hence they are isomorphic. (See \cite[a paragraph below Lemma 5.2]{MR3120578} for more detail, 
in which a similar argument is given.)
\end{Rem}

For $i \in \hI$, let $\hfn_i$ be the nilradical of $\hfp_i$. 
More explicitly, $\hfn_i$ is defined as follows:
\[ \hfn_i = \bigoplus_{\ga \in \gD_+\setminus \{\ga_i\}} \C e_\ga \oplus t\fg[t] \ (i \in I), \ 
   \hfn_0 = \fn_+ \oplus \bigoplus_{\ga \in \gD\setminus\{-\theta\}} \C (e_\ga \otimes t) \oplus (\fh\otimes t)\oplus
   t^2\fg[t].
\]
The following lemma is useful to determine defining relations of modules constructed using $F_w$'s.
For the proof, see \cite[Lemma 5.3]{MR3120578}.

\begin{Lem}\label{Lem:proceed}
  Let $V$ be an integrable $\hfg$-module, $T$ a finite-dimensional $\hfb$-submodule of $V$,
  $i \in \hI$ and $\xi \in \hP$ such that $\langle h_i, \xi \rangle \ge 0$.
  Assume that the following conditions hold:\\
  {\normalfont(i)}
    $T$ is generated by a weight vector $v \in T_\xi$ satisfying $e_i v=0$.\\
  {\normalfont(ii)}
    There is an $\ad(e_i)$-invariant left $U(\hfn_i)$-ideal $\mathcal{I}$ such that 
    \[ \mathrm{Ann}_{U(\hfn_+)}v = U(\hfn_+)e_i + U(\hfn_+)\mathcal{I}.
    \]
  {\normalfont(iii)} We have $\ch F_i T = \mathcal{D}_i \ch T$, where $\mathrm{ch}$ denotes the character with respect to 
   $\hfh$, and $\mathcal{D}_i$ is the Demazure operator defined by
   \[ \mathcal{D}_i(f) = \frac{f-e^{-\ga_i}s_i(f)}{1-e^{-\ga_i}}.
   \]
  Let $v' = f_i^{\langle h_i, \xi \rangle} v$. Then we have
  \[ \mathrm{Ann}_{U(\hfn_+)} v'= U(\hfn_+)e_i^{\langle h_i, \xi \rangle+1} + U(\hfn_+)\Phi_i( \mathcal{I}).
  \]
\end{Lem}

\subsection{Presentation by Chari and Venkatesh}\label{Subsection:Chari-Venkatesh}

Following \cite{MR3296163}, we introduce some notations.
For $r,s \in\Z_{\ge 0}$, let
\[ \mathbf{S}(r,s)=\Big\{(b_j)_{j\ge 0}\Bigm| b_j \in \Z_{\ge 0}, \ \sum_j b_j = r, \ \sum_j jb_j = s\Big\}.
\]
Note that $\mathbf{S}(0,s) = \emptyset$ if $s > 0$, and if $(b_j)_{j\ge 0} \in \mathbf{S}(r,s)$ then
$b_j = 0$ for $j>s$.
For $x \in \fg$ and $r,s\in \Z_{\ge 0}$, define a vector $x(r,s) \in U(\fg[t])$ by
\[ x(r,s) = \sum_{(b_j)_{j\ge 0}\in \mathbf{S}(r,s)} (x\otimes 1)^{(b_0)}(x\otimes t)^{(b_1)} \cdots 
   (x\otimes t^s)^{(b_s)},
\]
where for $X \in \fg[t]$, $X^{(b)}$ denotes the divided power $X^b/b!$.
We understand $x(r,s) = 0$ if $\mathbf{S}(r,s)=\emptyset$. 
For $\ga \in \gD_+$, define Lie subalgebras $\mathfrak{sl}_{2,\ga}$ and $\fb_{\ga}$ of $\fg$ by 
\[ \msl_{2,\ga}= \C e_\ga \oplus \C  h_\ga \oplus \C f_{\ga},\ \ \ \fb_{\ga}=\C e_\ga\oplus \C h_\ga.
\]
We also define a Lie subalgebra $\hfm_\ga$ of $\msl_{2,\ga}[t]$ by
\[ \hfm_\ga = t\msl_{2,\ga}[t] \oplus \C f_\ga.
\]
By \cite{MR502647} (see also \cite[Lemma 1.3]{MR1850556}), we have
\begin{align}\label{eq:Garland}
 (e_\ga\otimes t)^{(s)}f_{\ga}^{(r+s)} - &(-1)^sf_{\ga}(r,s)\in U(\hfm_\ga)t\fb_\ga[t].
\end{align}
For $k \in \Z_{\ge 0}$, let $\mathbf{S}(r,s)_k$ (resp.\ ${}_k\mathbf{S}(r,s)$) be the subset of $\mathbf{S}(r,s)$ consisting
of elements $(b_j)_{j\ge0}$, satisfying
\[ b_j =0 \ \text{ for } \ j \ge k \ \ \ (\text{resp.\ } b_j = 0 \ \text{ for } \ j<k).
\]
For $x\in \fg$, define a vector $x(r,s)_k$ and ${}_kx(r,s)$ by
\begin{align*}
 x(r,s)_k &= \sum_{(b_j)_{j\ge 0}\in \mathbf{S}(r,s)_k} (x\otimes 1)^{(b_0)}(x\otimes t)^{(b_1)} \cdots 
  (x\otimes t^{k-1})^{(b_{k-1})},\\ 
 {}_kx(r,s)&= \sum_{(b_j)_{j\ge0}\in {}_k\mathbf{S}(r,s)} (x\otimes t^k)^{(b_k)}(x\otimes t^{k+1})^{(b_{k+1})} \cdots 
  (x\otimes t^{s})^{(b_s)}.
\end{align*}
The following was proved in \cite{MR3296163}.

\begin{Lem}\label{Lem:Chari-Venkatesh}
 {\normalfont(i)} Let $x \in \fg$. If $r,s,k\in \Z_{>0}$ and $K \in \Z_{\ge 0}$ satisfy $r+s\ge kr+K$, then
 \[ x(r,s)={}_kx(r,s) + \sum_{(r',s')} x(r-r',s-s')_k\cdot {}_kx(r',s'),
 \]
 where the sum is over all pairs $r',s' \in \Z_{\ge 0}$ such that $r' < r, s'\le s$ and $r'+s' \ge kr' + K$.\\
 \noindent{\normalfont(ii)} Let $\ga \in \gD_+$, $V$ be an $\mathfrak{sl}_{2,\ga}[t]$-module, $v \in V$
  and $K \in \Z_{\ge 0}$.
  Assume that $e_\ga\otimes \C[t]$ and $h_\ga\otimes t\C[t]$ act trivially on $v \in V$.
  Then, 
  \[ (e_\ga \otimes t)^{s}f_{\ga}^{r+s}v = 0 \text{ for all } r,s \in \Z_{>0} \text{ with }
      r+s \ge 1 + kr + K \text{ for some } k\in \Z_{>0}
  \]
  if and only if
  \[ {}_kf_{\ga}(r,s)v=0 \ \text{ for all } r,s,k \in \Z_{>0} \text{ with } r+s \ge 1 + kr + K.
  \]
\end{Lem}

The following proposition plays an important roll in the next section.

\begin{Prop}\label{Prop2}
 Let $\ga \in \gD_+$. If $r,s,k \in \Z_{>0}$ and $K \in \Z_{\ge 0}$ satisfy $r+s \ge kr + K$,
 then we have 
 \[ \big[ e_\ga, \ {}_kf_{\ga}(r,s)\big] \in \sum_{(r',s')}U(t\mathfrak{sl}_{2,\ga}[t]){}_k\bfa(r',s') + 
    U(t\msl_{2,\ga}[t])t\fb_\ga[t],
 \]
 where the sum is over all pairs $r',s' \in \Z_{> 0}$ such that $r'+s'\ge kr'+K$.
\end{Prop}

\begin{proof}
 First we introduce some notation.
 We write $\mathfrak{f}_\ga = \C f_\ga$ here.
 Define Lie subalgebras $\hfm_\ga^h$, $\mathfrak{f}_\ga[t]_{< k}$ and $\ff_\ga[t]_{<k}^h$ by
 \[ \hfm_\ga^h = \hfm_\ga \oplus \C h_\ga, \ \ \ \ff_\ga[t]_{<k} = \bigoplus_{j=0}^{k-1} \C (f_\ga \otimes t^j), \ \ \ 
    \ff_{\ga}[t]_{<k}^h = \ff_{\ga}[t]_{<k} \oplus \C h_\ga.
 \]
 Since 
 \[ \hfm_\ga^h = \ff_{\ga}[t]_{<k}^h \oplus t^k\ff_{\ga}[t] \oplus t\fb_\ga[t],
 \]
 we have by the PBW theorem that
 \[ U(\hfm_\ga^h) = \ff_{\ga}[t]_{<k}^hU(\hfm_\ga^h) \oplus U\big(t^k\ff_{\ga}[t]\oplus t\fb_\ga[t]\big).
 \]
 Denote by $p$ the projection 
 \[ U(\hfm_\ga^h) \twoheadrightarrow U\big(t^k\ff_\ga[t] \oplus t\fb_\ga[t]\big)
 \]
 with respect to this decomposition.  
 It follows from Lemma \ref{Lem:Chari-Venkatesh} (i) that
 \begin{equation}\label{eq:image_of_p}
  p\big(\bfa(r',s')\big) = {}_k\bfa(r',s').
 \end{equation}
 Denote by $\mathcal{I}$ the left $U(t\msl_{2,\ga}[t])$-ideal in the assertion.

 Now we begin the proof of the proposition.
 By (\ref{eq:Garland}), it follows that
 \[ \big[e_\ga,\ (e_\ga\otimes t)^{(s)}f_\ga^{(r+s)}\big] - 
    (-1)^s\big[e_\ga, \ \bfa(r,s)\big] \in U(\hfm_\ga^h)t\fb_\ga[t].
 \]
 By applying $p$ to this, we have
\begin{align}\label{eq:containment2}
 p\Big(\big[e_\ga,\ (e_\ga\otimes t)^{(s)}f_\ga^{(r+s)}\big]\Big) - 
    (-1)^sp&\Big(\big[e_\ga, \ \bfa(r,s)\big]\Big)\in U\big(t^k\ff_\ga[t]\oplus t\fb_\ga[t]\big)t\fb_\ga[t]\subseteq \mathcal{I}.
\end{align}
 The following calculation is elementary:
\begin{align*}
 \big[e_\ga, \ (e_\ga\otimes t)^{(s)}f_\ga^{(r+s)}\big] = (e_\ga\otimes t)^{(s)}\big[e_\ga, \, f_\ga^{(r+s)}\big]
   = (h_\ga+r-s-1)(e_\ga\otimes t)^{(s)}f_\ga^{(r+s-1)}.
\end{align*}
Note that the pair $(r-1,s)$ satisfies the condition $(r-1)+s \ge k(r-1)+K$ since $k \in \Z_{>0}$.
By (\ref{eq:Garland}), the above equality implies
\begin{align}\label{eq:containment}
 p\Big(\big[e_\ga, \ (e_\ga\otimes t)^{(s)}f_\ga^{(r+s)}\big]\Big) \in&\ p\Big(\C(e_\ga\otimes t)^{(s)}
  f_\ga^{(r+s-1)}\Big)\nonumber\\
 \subseteq&\ p\Big(\C \bfa(r-1,s) + U(\hfm_\ga)t\fb_\ga[t]\Big)\\ =&\ \C\, {}_k\bfa(r-1,s) + 
   U\big(t^k\ff_\ga[t] \oplus t\fb_\ga[t]\big)t\fb_\ga[t] \subseteq \mathcal{I},\nonumber
\end{align}
where the equality holds by (\ref{eq:image_of_p}).
On the other hand, we have by Lemma \ref{Lem:Chari-Venkatesh} (i) that
\begin{align}\label{eq:pro_equal}
 p\Big(\big[e_\ga, \ \bfa(r,s)\big]\Big) = p \Big( \big[e_\ga,\ {}_k\bfa(r,s)\big]\Big) + \sum_{(r',s')} 
 p\Big(\big[ e_\ga, \ \bfa(r-r',s-s')_k\cdot {}_k\bfa(r',s')\big]\Big).
\end{align} 
Since $\big[e_\ga,\ {}_k\bfa(r,s)\big] \in U(t^k\mathfrak{sl}_{2,\ga}[t])$,
it follows that $p\Big(\big[e_\ga, \ {}_k\bfa(r,s)\big]\Big) = \big[e_\ga, \ {}_k\bfa(r,s)\big]$, and
\begin{align*}
 p\Big(\big[ e_\ga, \ \bfa(r&-r',s-s')_k\cdot {}_k\bfa(r',s')\big]\Big)\\
 \hspace{-5pt}=&\, p\Big(\big[e_\ga, \ \bfa(r-r',s-s')_k\big]{}_k\bfa(r',s')\Big) + p\Big(\bfa(r-r',s-s')_k 
 \big[e_\ga, \ {}_k\bfa(r',s')\big]\Big)\\
 \hspace{-5pt}\in&\, p\Big(U(\hfm_\ga^h){}_k\bfa(r',s')\Big) + 0 = U\big(t^k\ff_\ga[t] \oplus t\fb_\ga[t]\big){}_k\bfa(r',s') 
 \subseteq \mathcal{I}.
\end{align*}
Hence (\ref{eq:pro_equal}) implies
\begin{align*}
 p\Big(\big[e_\ga, \ \bfa(r,s)\big]\Big) - \big[e_\ga, \ {}_k\bfa(r,s)\big]
 \in \mathcal{I}.
\end{align*}
Now $\big[e_\ga, \ {}_k\bfa(r,s)\big] \in \mathcal{I}$ follows from this, together with (\ref{eq:containment2}) and
(\ref{eq:containment}). 
The proof is complete.
\end{proof}

\section{Main theorem and proof}\label{Section:proof}

Let $m \in I$ and $\bm{\ell} = (\ell_1\ge\cdots\ge\ell_p)$ be a partition,
and denote by $V_m(\bm{\ell})$ the fusion product $V(\ell_1\varpi_m) * \cdots * V(\ell_p\varpi_m)$.
Set $L_i = \ell_i + \cdots +\ell_p$ for $1\le i \le p$, and $L_i = 0$ for $i>p$.
As mentioned in the introduction, the main theorem of this note is the following.

\begin{Thm}\label{Thm:Main2}
 The fusion product $V_m(\bm{\ell})$ is isomorphic to the 
 $\fg[t]$-module generated by a vector $v$ with relations
 \begin{align*}
  \fn_+[t]v=0, \ \ \ (h\otimes t^s) v&= \gd_{s0}L_1\langle h,\varpi_m\rangle v \text{ for } 
  h \in \fh, \ s \in \Z_{\ge 0},\nonumber\\
  \big(f_\ga\otimes \C[t]\big)v&= 0 \text{ for } \ga \in \gD_+ \text{ with } \langle h_\ga,\varpi_m\rangle = 0,\nonumber\\
  f_\ga^{L_1+1}v&=0 \text{ for } \ga \in \gD_+ \text{ with } \langle h_\ga, \varpi_m\rangle =1, \nonumber\\
  (e_\ga\otimes t)^{s}f_\ga^{r+s}v&=0 \text{ for } \ga \in \gD_+, r,s \in \Z_{>0} \text{ with } \langle h_\ga, \varpi_m 
   \rangle = 1,\nonumber\\
  &\hspace{40pt} r+s \ge 1+kr + L_{k+1} \text{ for some } k\in\Z_{>0}.
 \end{align*}
\end{Thm}

\begin{Rem}\normalfont
 In \cite{MR3296163}, the authors have introduced a collection of $\fg[t]$-modules $V(\bm{\xi})$ 
 (with $\fg$ a general simple Lie algebra) indexed by 
 a $|\gD_+|$-tuple of partitions $\bm{\xi} = (\xi^\ga)_{\ga \in \gD_+}$ satisfying $|\xi^\ga|=\langle h_\ga,\gl\rangle$
 for some $\gl \in P_+$.
 In their terminology, the theorem says that $V_m(\bm{\ell})$ is isomorphic to $V(\bm{\xi})$ where 
 $\bm{\xi} = (\xi^\ga)_{\ga\in \gD_+}$ with
 \[ \xi^\ga = \begin{cases} \bm{\ell} & \text{if }\langle h_\ga,\varpi_m\rangle = 1,\\
                            0         & \text{if }\langle h_\ga,\varpi_m\rangle = 0.
              \end{cases}
 \]
\end{Rem}

The proof of the theorem will occupy the rest of this paper.
Fix $m \in I$ and $\bm{\ell}$ from now on.
By Theorem \ref{Thm:realization}, we have
\begin{align}\label{eq:isom}
  V_m(\bm{\ell})
  \cong F_{t_{-\varpi_{m^*}}}\Big(\C_{(\ell_1-\ell_2)\gL_0}\otimes \cdots \otimes F_{t_{-\varpi_{m^*}}}
  \big(\C_{(\ell_{p-1}-\ell_p)\gL_0} 
  \otimes F_{t_{-\varpi_{m^*}}}\C_{\ell_p\gL_0}\big)\!\cdots\!\Big)
\end{align}
as $\hfb'$-modules.
We shall determine defining relations of the right-hand side recursively.
In the sequel, we write $\tau= \tau_m$ and $\gs = w_0w_{m,0}$ (see Subsection \ref{Subsection:preliminary}).
Note that 
\[ \gs(\varpi_m) = w_0(\varpi_m) = -\varpi_{m^*} \ \text{ and } \ t_{-\varpi_{m^*}} = \gs t_{\varpi_m}\gs^{-1} = \gs\tau
\]
hold.
Let $\gs=s_{i_{\ell(\gs)}}\cdots s_{i_2}s_{i_1}$ be a reduced expression of $\gs$, and set $\gs_j = 
s_{i_j}\cdots s_{i_2}s_{i_1}$ for $0\le j \le \ell(\gs)$.
For $a \in \{0,\pm 1\}$, define a subset $\gD[a] \subseteq \gD$ by
\[ \gD[a] = \{ \ga \in \gD\mid \langle h_\ga,\varpi_m\rangle = a\}.
\]
Note that $\gD[\pm 1] \subseteq \pm \gD_+$, and 
\begin{equation}\label{eq:small_remark}
  \ga \in \gD[a] \text{ if and only if } \langle \gs(h_\ga),\varpi_{m^*}\rangle 
  = -a.
\end{equation}
We also write $\gD[\ge\! 0] = \gD[0] \sqcup \gD[1]$, etc.
It should be noted that, since $\gs$ is the shortest element such that $\gs(\varpi_m) = -\varpi_{m^*}$,
for every $1 \le j \le \ell(\gs)$ we have 
\begin{equation}\label{eq:positive2}
 \langle h_{i_j},\gs_{j-1}(\varpi_m)\rangle = 1 \text{ and } \gs_{j-1}^{-1}(\ga_{i_j}) \in \gD[1].
\end{equation}
Define a parabolic subalgebra $\fp_{\varpi_m}$ of $\fg$ by
\[ \fp_{\varpi_m} = \bigoplus_{\ga \in \gD[\ge 0]} \C e_\ga \oplus \fh = \bigoplus_{\ga \in \gD[0] \cap \gD_+} \C f_\ga
   \oplus \fb.
\]
For $1\le q \le p$ and $0\le j \le \ell(\gs)$, let $V(q,j)$ be the $\hfb$-module
\[ F_{\gs_j\tau}\Big(\C_{(\ell_q-\ell_{q+1})\gL_0} \otimes F_{t_{-\varpi_{m^*}}}\Big(\cdots\otimes F_{t_{-\varpi_{m^*}}}
   \big(\C_{(\ell_{p-1}-\ell_p)\gL_0} 
   \otimes F_{t_{-\varpi_{m^*}}}\C_{\ell_p\gL_0}\big)\!\cdots\!\Big)\!\Big).
\]

\begin{Prop}\label{Prop*}
 For every $q$ and $j$, there exists a nonzero vector $v_{q,j}$ in $V(q,j)$ whose $\hfh_{\mathrm{cl}}$-weight is 
 $L_q\gs_j(\varpi_m)+\ell_q\gL_0$, such that $V(q,j)$ is generated by $v_{q,j}$ as a $\hfb'$-module and 
 \begin{align}\label{eq:annihilators}
  \mathrm{Ann}_{U(\hfn_+)}v_{q,j}=  \sum_{\begin{smallmatrix}\ga \in \gD[-1] \\ \gs_j(\ga) \in \gD_+\end{smallmatrix} }
   U(\hfn_+)& e_{\gs_j(\ga)}^{L_q+1} + \sum_{\begin{smallmatrix}\ga \in \gD[\ge0] \\ \gs_j(\ga) \in \gD_+\end{smallmatrix} }
   U(\hfn_+) e_{\gs_j(\ga)} \\&+\sum_{\ga\in\gD[-1]}\sum_{(r,s,k)}
   U(\hfn_+){}_ke_{\gs_j(\ga)}(r,s) +U(\hfn_+)\Phi_{\gs_j}\big(t\fp_{\varpi_m}[t]\big),  \nonumber
 \end{align}
 where the sum for $(r,s,k)$ is over all $r,s,k \in \Z_{>0}$ such that $r+s \ge 1+ kr + L_{k+q}$.
\end{Prop}

For a while we assume this proposition, and give a proof to Theorem \ref{Thm:Main2}.
Denote by $T_q$ the running set of $(r,s,k)$ in (\ref{eq:annihilators}), that is,
\[ T_q=\big\{(r,s,k)\in \Z_{>0}^3\bigm| r+s \ge 1 + kr + L_{k+q}\big\}.
\]
Since $\langle h_\ga, \varpi_{m^*}\rangle = -\langle h_{\gs^{-1}(\ga)}, \varpi_m\rangle$, we see that
for $\ga \in \gD_+$, $\gs^{-1}(\ga) \in \gD[-1]$ is equivalent to $\langle h_\ga,\varpi_{m^*}\rangle = 1$.
Hence (\ref{eq:isom}) and Proposition \ref{Prop*} with $q = 1$ and $j = \ell(\gs)$ imply that
there exists a nonzero vector $v'$ in $V_m(\bm{\ell})$ whose $\fh$-weight is $-L_1\varpi_{m^*}$, such that $V_m(\bm{\ell})$ is 
generated by $v'$ and
\begin{align}\label{eq:annihilators_of_1}
 \mathrm{Ann}_{U(\hfn_+)}v' =  \sum_{\ga \in \gD_+}
   U(\hfn_+)& e_{\ga}^{L_1\langle h_\ga, \varpi_{m^*}\rangle+1}\\
  &+ \sum_{\ga \in \gD[-1]}\sum_{(r,s,k) \in T_1}U(\hfn_+){}_ke_{\gs (\ga)}(r,s)
   +U(\hfn_+)\Phi_{\gs}(t\fp_{\varpi_m}[t]),\nonumber
\end{align}
where the first summation in the right-hand side is obtained using (\ref{eq:small_remark}).
$V_m(\bm{\ell})$ being a finite-dimensional $\fg$-module, $\Phi_{w_0}^{V_m(\bm{\ell})}$ is defined.
Set $v''= \Phi^{V_m(\bm{\ell})}_{w_0}(v') \in V_m(\bm{\ell})_{L_1\varpi_m}$, 
and $\hfm_+=\Phi_{w_0}(\hfn_+) = \fn_-[t] \oplus t\fb[t]$.
Since each $\gD[a]$ is stable by $w_0\gs  = w_{m,0}$ and $\gD[-1] = -\gD[1]$, it follows that
\begin{align*}
 \mathrm{Ann}_{U(\hfm_+)}v'' =  \sum_{\ga \in \gD_+} U(\hfm_+)f_\ga^{L_1\langle h_\ga,\varpi_m\rangle+1}
  + \sum_{\ga \in \gD[1]}\sum_{(r,s,k)\in T_1}U(\hfm_+){}_kf_{\ga}(r,s)
   +U(\hfm_+)t\fp_{\varpi_m}[t].
\end{align*}
Let $M$ be the $\fg[t]$-module generated by a vector $v$ with relations in Theorem \ref{Thm:Main2}.
By Lemma \ref{Lem:Chari-Venkatesh} (ii), $v$ satisfies
\[ {}_kf_\ga(r,s)v=0 \ \text{ for } \ga \in \gD[1], (r,s,k) \in T_1.
\]
Then we see from the above description of $\mathrm{Ann}_{U(\hfm_+)}v''$ that 
there exists a surjective $\hfm_+$-module homomorphism
from $V_m(\bm{\ell})$ to $M$ mapping $v''$ to $v$.
On the other hand, since
\[ V_m(\bm{\ell}) \cong V(\ell_1\varpi_m) \otimes \cdots \otimes V(\ell_p\varpi_m)
\] 
as a $\fg$-module,
we have $V_m(\bm{\ell})_{\mu} = 0$ if $\mu > L_1\varpi_{m}$, which implies $\fn_+v''=0$.
Then again by Lemma \ref{Lem:Chari-Venkatesh} (ii), $v''$ satisfies $(e_\ga \otimes t)^sf_\ga^{r+s}v'' = 0$ for 
$\ga \in \gD[1]$ and $r,s$ with $(r,s,k) \in T_1$ for some $k \in \Z_{>0}$,
and we also see that there exists a surjective $\fg[t]$-module homomorphism from $M$ to $V_m(\bm{\ell})$ mapping $v$ to $v''$.
Hence $V_m(\bm{\ell}) \cong M$ holds, and the theorem is proved.

The rest of this paper is devoted to prove Proposition \ref{Prop*}.
Define a left $U(\hfn_+)$-ideal $\mathcal{I}(q,j)$ by the right-hand side of (\ref{eq:annihilators}).
We prove the assertion by the induction on $(q,j)$. 
When $q=p$ and $j = 0$,  
\[ V(p,0) = F_\tau\C_{\ell_p\gL_0} \cong \C_{\ell_p(\varpi_m+\gL_0)}
\]
is a $1$-dimensional module with $\hfh_{\mathrm{cl}}$-weight $\ell_p(\varpi_m+\gL_0)$ on which $\hfn_+$ acts trivially.
Hence in order to verify the assertion in this case, it suffices to show that $\mathcal{I}(p,0) = U(\hfn_+)$. 
The containment $\mathcal{I}(p,0) \subseteq U(\hfn_+)$ is obvious,
and $\fn_+ + t\fp_{\varpi_m}[t] \subseteq \mathcal{I}(p,0)$ is easily seen.
Moreover since $L_{1+p} = 0$, $(1,s,s) \in T_p$ for every $s\in \Z_{>0}$,
and hence we have 
\[ {}_se_\ga(1,s) = e_\ga \otimes t^s \in \mathcal{I}(p,0) \ \text{ for } \ga \in \gD[-1], \ s \in \Z_{>0}.
\]
Hence $U(\hfn_+) \subseteq \mathcal{I}(p,0)$ holds.

Next we shall prove that, if the assertion for $(q,j-1)$ holds, then that for
$(q,j)$ also holds.
We write $i = i_j$ for short.
We have $V(q,j) = F_{i}V(q,j-1)$,
and the $\hfh_{\mathrm{cl}}$-weight of $v_{q,j-1}$ is $L_q\gs_{j-1}(\varpi_m)+\ell_q\gL_0$.
Moreover $e_iv_{q,j-1}=0$ holds by (\ref{eq:positive2}).
Set $v_{q,j} = f_i^{L_q}v_{q,j-1}$.
Since $V(q,j)$ is a submodule of an integrable module, 
it follows from the representation theory of $\mathfrak{sl}_2$ that
\[ v_{q,j} \neq 0, \ \ \ f_iv_{q,j} = 0,\ \text{ and } \ e_i^{L_q}v_{q,j} 
   \in \C^\times v_{q,j-1}.
\]
Hence we have
\[ V(q,j)=F_iV(q,j-1) = U(\hfp_i)v_{q,j-1} = U(\hfp_i)v_{q,j} = U(\hfb')v_{q,j},
\]
and the cyclicity of $V(q,j)$ is proved.
Moreover it is obvious that the $\hfh_{\mathrm{cl}}$-weight of $v_{q,j}$ is $L_{q}\gs_j(\varpi_m)+\ell_q\gL_0$.
It remains to prove $\mathrm{Ann}_{U(\hfn_+)}(v_{q,j}) = \mathcal{I}(q,j)$.
Let $\mathcal{J}$ be the left $U(\hfn_{i})$-ideal defined by
\begin{align*}
  \mathcal{J}=  \sum_{\begin{smallmatrix}\ga \in \gD[-1]\\ \gs_{j-1}(\ga) \in \gD_+\end{smallmatrix}}
   U(\hfn_i)&e_{\gs_{j-1}(\ga)}^{L_q+1} + \sum_{\begin{smallmatrix}\ga \in \gD[\ge0]\\ \gs_{j-1}(\ga) \in 
   \gD_+\setminus \{\ga_i\}
   \end{smallmatrix}}
   U(\hfn_i)e_{\gs_{j-1}(\ga)}\\&+\sum_{\ga\in\gD[-1]}\sum_{(r,s,k)\in T_q}
   U(\hfn_i){}_ke_{\gs_{j-1}(\ga)}(r,s) +U(\hfn_i)\Phi_{\gs_{j-1}}\big(t\fp_{\varpi_m}[t]).
\end{align*}
By the induction hypothesis we have
\begin{align*}
 \mathrm{Ann}_{U(\hfn_+)}v_{q,j-1} =\mathcal{I}(q,j-1)= U(\hfn_+)e_{i}+U(\hfn_+)\mathcal{J}.
\end{align*}
It suffices to show that $V(q,j-1)$, $v_{q,j-1}$ and $\mathcal{J}$ satisfy the conditions (i)--(iii) in Lemma
\ref{Lem:proceed}.
Indeed if they satisfy the conditions, it follows from the lemma that
\begin{align*}
 \mathrm{Ann}_{U(\hfn_+)}v_{q,j} = U(\hfn_+)e_i^{L_{q}+1} + U(\hfn_+)\Phi_i(\mathcal{J})= \mathcal{I}(q,j),
\end{align*}
as required.
The condition (i) follows from the induction hypothesis, and the condition (iii) is proved by \cite[Theorem 5]{MR1887117},
or \cite[Corollary 2.13 and Lemma 3.2(ii)]{MR2964614}.
In order to show the condition (ii) we need to prove that $\mathcal{J}$ is $\mathrm{ad}(e_i)$-invariant.
For that, we first verify for $\ga \in \gD \setminus \{-\ga_i\}$ that
\[ \ad(e_i)U\big(t\fg_\ga[t]\big) \subseteq \mathcal{J}.
\]
If $\ga + \ga_i \notin \gD$, this is obvious.
Assume that $\ga + \ga_i \in \gD$. 
Then it follows from (\ref{eq:positive2}) that
\[ \ga + \ga_i = \gs_{j-1}\big(\gs_{j-1}^{-1}(\ga) + \gs_{j-1}^{-1}(\ga_i)\big) \in \gs_{j-1}\big(\gD[\ge\!0]\big).
\]
Since $[e_{\ga+\ga_i},e_\ga] = 0$ holds, this implies
\[ \ad(e_i)U\big(t\fg_\ga[t]\big) \subseteq U(t\fg_\ga[t])t\fg_{\ga+\ga_i}[t] \subseteq \mathcal{J},
\] 
as required.
In a similar manner, $\ad(e_i)U\big(\C e_\ga\big) \subseteq \mathcal{J}$ for $\ga \in \gD_+$ is also proved.
In addition, $\ad(e_i)\big(t\fh[t]\big)= t\fg_{\ga_i}[t] \subseteq \mathcal{J}$ follows from (\ref{eq:positive2}).
Now combining these facts with Proposition \ref{Prop2}, $\ad(e_i)\big(\mathcal{J}\big) \subseteq \mathcal{J}$ is proved.

Finally it remains to prove that the assertion for $\big(q+1,\ell(\gs)\big)$ implies that for $(q,0)$.
Note that 
\begin{align*}
 V(q,0) = F_\tau \Big(\C_{(\ell_q-\ell_{q+1})\gL_0} \otimes V\big(q+1,\ell(\gs)\big)\Big)
  \cong \C_{(\ell_q-\ell_{q+1})(\varpi_m+\gL_0)} \otimes (\tau^{-1})^*V\big(q+1,\ell(\gs)\big).
\end{align*}
Let $z$ be a basis of $\C_{(\ell_q-\ell_{q+1})(\varpi_m+\gL_0)}$ and set
$v_{q,0} = z \otimes (\tau^{-1})^*v_{q+1,\ell(\gs)}$,
where $(\tau^{-1})^*v_{q+1,\ell(\gs)}$ is the image of $v_{q+1,\ell(\gs)}$ under the linear isomorphism 
$V\big(q+1,\ell(\gs)\big) \to (\tau^{-1})^*V\big(q+1,\ell(\gs)\big)$.
By (\ref{eq:rule_of_tau}), the $\hfh_{\mathrm{cl}}$-weight of $v_{q,0}$ is 
\begin{align*}
 (\ell_q-\ell_{q+1})(\varpi_m+\gL_0&) + \tau(-L_{q+1}\varpi_{m^*}+\ell_{q+1}\gL_0)\\ &= (\ell_q-\ell_{q+1})(\varpi_m+\gL_0)
   + \tau\big(-L_{q+1}(\varpi_{m^*}+\gL_0) + (\ell_{q+1}+L_{q+1})\gL_0\big)\\
  &= (\ell_q-\ell_{q+1})(\varpi_m+\gL_0) + \big(-L_{q+1}\gL_0+ (\ell_{q+1}+L_{q+1})(\varpi_m+\gL_0)\big)\\
  &= L_q\varpi_m + \ell_q\gL_0.
\end{align*}
Moreover since $\hfn_+$ acts trivially on $z$, we have
\begin{align}\label{eq:final_ann}
 \mathrm{Ann}_{U(\hfn_+)}(v_{q,0}) = \tau\Big(\mathrm{Ann}_{U(\hfn_+)}(v_{q+1,\ell(\gs)})\Big)
 =\tau\Big(\mathcal{I}\big(q+1,\ell(\gs)\big)\Big).
\end{align}
By noting $\tau \gs = t_{\varpi_m}$, we see that 
\begin{equation}\label{eq:tau2}
 \tau(e_{\gs(\ga)}\otimes t^s) = e_\ga \otimes t^{s-\langle h_\ga,\varpi_m\rangle} \ \text{ for } \ga \in \gD, 
 s \in \Z_{\ge 0}.
\end{equation}

\begin{Lem}\label{Lem:elementary_lemma}
 For $\ga \in \gD[-1]$ and $r,s,k \in \Z_{>0}$ we have
 \[ \tau\big({}_ke_{\gs(\ga)}(r,s)\big) = {}_{k+1}e_{\ga}(r,s+r).
 \]
\end{Lem}

\begin{proof}
 It is easily seen that the map 
 \[ {}_k\mathbf{S}(r,s) \ni (b_j)_{j\ge 0} \to (b_j')_{j\ge0}\in{}_{k+1}\mathbf{S}(r,s+r)
 \] 
 defined by $b_j' = b_{j-1}$ is bijective.
 Then the assertion is proved from (\ref{eq:tau2}).
\end{proof}

Using (\ref{eq:tau2}) and the above lemma, we see that
\begin{align}\label{eq:tauI}
 \tau\Big(\mathcal{I}\big(q+1,\ell(\gs)\big)\Big) = 
 \sum_{\ga \in \gD[1]} U(\hfn_+&)\Big\{\C(f_\ga\otimes t)^{L_{q+1}+1}
 \\ +&\sum_{(r,s,k)\in T_{q+1}}\C\, {}_{k+1}f_\ga(r,s+r)\Big\}
 + U(\hfn_+)\big(\fn_+ + t\fp_{\varpi_m}[t]\big).\nonumber 
\end{align}
On the other hand, we have
\begin{align}\label{eq:Iq0}
 \mathcal{I}(q,0) = \sum_{\ga \in \gD[1]} \sum_{(r,s,k)\in T_q}U(\hfn_+){}_{k}f_\ga(r,s)
 + U(\hfn_+)\big(\fn_+ + t\fp_{\varpi_m}[t]\big).
\end{align}
It is easily checked that 
\[ \big\{(r,s+r,k+1)\bigm| (r,s,k) \in T_{q+1}\big\} = \{(r,s,k) \in T_q\mid k>1\},
\]
which implies
\begin{equation}\label{eq:easy_equal}
  \sum_{(r,s,k)\in T_{q+1}}U(\hfn_+) {}_{k+1}f_\ga(r,s+r) = \sum_{\begin{smallmatrix} (r,s,k)\in T_q \\ k > 1
  \end{smallmatrix}}U(\hfn_+) {}_{k}f_\ga(r,s).
\end{equation}
We shall prove for each $\ga \in \gD[1]$ that
\begin{align}\label{eq:final1}
 (f_\ga\otimes \,t)^{L_{q+1}+1}&\in \sum_{(r,s,k)\in T_q}U(\hfn_+){}_{k}f_\ga(r,s), \ \text{ and}\\ 
 \label{eq:final2}
 {}_1f_\ga(r,s) &\in U(\hfn_+)\Big(\C(f_\ga\otimes t)^{L_{q+1}+1} +  \fn_+[t] + t\fh[t]\Big)\ \text{ for } r,s 
 \text{ with } (r,s,1) \in T_q.
\end{align}
By comparing (\ref{eq:tauI}) and (\ref{eq:Iq0}) and using (\ref{eq:easy_equal}), we see that these imply 
\[ \tau\Big(\mathcal{I}\big(q+1,\ell(\gs)\big)\Big) = \mathcal{I}(q,0),
\]
and then by (\ref{eq:final_ann}) we have $\mathrm{Ann}_{U(\hfn_+)}(v_{q,0}) = \mathcal{I}(q,0)$,
which completes the proof.
Setting $r=s=L_{q+1}+1$, we have
\[ {}_1f_\ga(r,s) = (f_\ga\otimes t)^{L_{q+1}+1},
\]
and hence (\ref{eq:final1}) follows.
Assume that $r,s$ satisfy $(r,s,1) \in T_q$, which implies $s \ge L_{q+1}+1$.
Since ${}_1f_\ga(r,s) = 0$ if $r >s$, we may assume $r\le s$.
If we apply the Lie algebra automorphism $\tau\circ\Phi_{\gs}$ to (\ref{eq:Garland}) with $s$ replaced by $s-r$, we have
from Lemma \ref{Lem:elementary_lemma} that
\[ e_{\ga}^{(s-r)}(f_\ga\otimes t)^{(s)} - (-1)^{s-r}{}_1f_\ga(r,s) \in U(\hfn_+)\big(\fn_+[t] \oplus t\fh[t]\big).
\]
Hence (\ref{eq:final2}) also holds.

\section*{Acknowledgement}
The author is supported by JSPS Grant-in-Aid for Young Scientists (B) No.\ 25800006.

\def\cprime{$'$} \def\cprime{$'$}


\end{document}